\documentclass{m2an} 
\usepackage{amsmath, amssymb, latexsym}
\usepackage{srcltx}
\usepackage{epsfig}
\usepackage{pstricks,pst-all,multirow,graphicx}
\newtheorem{remark}{Remark}[section]
\newtheorem{theorem}{Theorem}

\newtheorem{corollary}[theorem]{Corollary}
\newtheorem{proposition}[theorem]{Proposition}
\usepackage[toc,page]{appendix}
\newtheorem{example}{Example}[section]

\def\b0{{\mathbf 0}}
 
\def\bc{{\mathbf c}}
\def\be{{\mathbf e}}
\def\bg{{\mathbf g}}\def\bff{{\mathbf f}}
 \def\bn{{\mathbf n}}
 
\def\bu{{\mathbf u}} \def\bv{{\mathbf v}}

\def\bx{{\mathbf x}} 
\def\bH{{\mathbf H}}
\def\bF{{\mathbf F}}

\def\bN{{\mathbf N}}
\def\bphi{\boldsymbol{\phi}}
\def\pd#1#2{\frac{\partial #1}{\partial#2}}

\def\OM{\Omega}
\def\RR{{\mathbb R}}

\def\cT {{\mathcal T}}

\newcommand{\wtilde}{\widetilde}
\newcommand{\what}{\widehat}
\newcommand{\Div}{\mathrm{div}}
\newcommand{\dstyle}{\displaystyle}
\newcommand{\tnorm}[1]{|\!|\!|#1|\!|\!|}

 \newcommand{\bsigma}{{\boldsymbol{\sigma}}}
\newcommand{\btau}{{\boldsymbol{\tau}}}

\newcommand{\bepsilon}{\boldsymbol{\epsilon}}

\def\bphi{\mbox{\boldmath{$\phi$}}}
\def\bsigma{\mbox{\boldmath{$\sigma$}}}

\makeatletter\@addtoreset{equation}{section}\makeatother

\def\pd#1#2{\frac{\partial #1}{\partial #2}}

\begin{document}
\title {A stabilized $P_1$-nonconforming immersed  finite element method for the interface elasticity problems}

\author{Do Y. Kwak}
\address{Korea Advanced Institute of Science and Technology,
Daejeon, Korea 305-701. {\tt email:kdy@kaist.ac.kr}. This  author is  supported by  NRF, No.2014R1A2A1A11053889.
}
\author{Sang W. Jin}
\address{Korea Advanced Institute of Science and Technology,
Daejeon, Korea 305-701. {\tt email:jinsangwon@kaist.ac.kr}.
}
\author{Dae H. Kyeong}
\address{Korea Advanced Institute of Science and Technology,
Daejeon, Korea 305-701. {\tt email:huff@kaist.ac.kr}.
}

\begin{abstract}
  We develop  a new finite element method for solving planar elasticity problems involving of heterogeneous materials
  with a mesh not necessarily aligning  with the interface of the materials. This method is  based on the `broken'  Crouzeix-Raviart $P_1$-nonconforming finite element method for elliptic interface problems \cite{Kwak-We-Ch}.
  To ensure the coercivity of the bilinear form arising from using the nonconforming finite elements, we
  add stabilizing terms as in the discontinuous Galerkin (DG) method  \cite{Arnold-IP},\cite{Ar-B-Co-Ma},\cite{Wheeler}.
The novelty of our method  is that we use meshes independent of the interface, so that the interface
may cut through  the elements. Instead, we modify the basis functions so that they satisfy the Laplace-Young condition along the interface of each element.  We prove optimal $H^1$ and divergence norm error estimates.
  Numerical experiments are carried out to demonstrate that the our method is optimal for various Lam\`e parameters $\mu$ and $\lambda$ and locking free as $\lambda\to\infty$.
\end{abstract}

\subjclass{AMS MOS 65N30, 74S05, 74B05}
\keywords{immersed finite element method, Crouzeix-Raviart finite element, elasticity problems, heterogeneous materials, Laplace-Young condition}

\maketitle \pagestyle{myheadings} \thispagestyle{plain}
\markboth{Do Y. Kwak, Sangwon Jin and Dae H. Kyeong} {
A stabilized $P_1$ immersed  finite element method for the interface elasticity problems}

\section{ Introduction}

Linear elasticity equation plays an important role in solid mechanics. In particular, when an elastic body
is occupied by heterogeneous materials having distinct Lam\`e parameters $\mu$ and $\lambda$, the governing equation
holds on each disjoint domain and certain jump conditions must be satisfied
along the interface of two materials \cite{A.P.Hansbo2004}.
This  kind  of problems involving composite materials is getting more and more attentions from both engineers and mathematicians in recent years, but efficient numerical schemes are not fully developed  yet. To solve  such equations numerically, one usually uses finite element methods with meshes aligned with the interface between two  materials.
However, such methods involve unstructured grids resulting in algebraic systems which involve more unknowns and
irregular data structure.

 Solving linear elasticity equation with finite element methods has been studied extensively and several methods have been developed, see \cite{Arnold-Wint},\cite{Brenner_Sung},\cite{Falk}  and references therein.
 For lower order methods, when $P_1$-conforming element method is applied, the so-called `locking phenomena' is observed when the material is nearly incompressible
 (\cite{Babuska-Suri1},\cite{Babuska-Suri2},\cite{BrezziFortin}). Brenner and Sung \cite{Brenner_Sung} showed that the
Crouzeix-Raviart (CR) $P_1$-nonconforming  element \cite{Crouzeix} does not lock on pure displacement problem.
 But one cannot  use this element  to a traction-boundary problem since it does not satisfy discrete Korn's inequality.
 A remedy was  recently suggested by Hansbo et al. \cite{P.Hansbo_Lar2002} who exploited the idea of discontinuous  Galerkin methods (\cite{Arnold-IP},\cite{Ar-B-Co-Ma},\cite{Wheeler}). By introducing a stabilizing term, they proved the convergence of a locking free $P_1$-nonconforming method for problems with traction boundary conditions.

 Solving problems with composite materials is more difficult. Since the Laplace-Young condition holds along  the interface, these problems exhibit a similar property as the traction boundary type problems, even if the Dirichlet boundary condition is imposed
 on the boundary of the whole domain. Thus the CR element may not work properly for such problems.
In the discussion of the above methods, meshes are assumed to be aligned with the interface.
We will resolve this problem  by adding  stabilizing terms for unaligned grids (See below).

 On the other hand, alternative methods which use meshes independent of interface,
  thus allowing the interface to cut through the elements, have been developed recently for diffusion problems. The motivations for using such meshes are : Easiness of grid generations, treatment of moving grids, especially time dependent problems, simple data structure of linear system, fast solvers, and so on.
There are two types of such methods in principle: One belongs to the extended finite element methods (XFEM) (\cite{Belytschko_Parimi}, \cite{Belytschko_Black}, \cite{Krysl-Bely},  \cite{Legrain}, \cite{Moes:99})
 and another belongs to  the immersed finite element methods (IFEM). (\cite{Chang_Kwak}, \cite{Ch-Kw-We}, \cite{Kwak-We-Ch},\cite{Li2004},\cite{Li2003})
In the XFEM type we need, in addition to the standard nodal basis functions, enriched basis functions obtained by truncating the shape functions along the interface cut so that three (six for planar elasticity problems) extra degrees of freedom are present per element.
But in the IFEMs, we do not require extra degrees of freedom, instead modify the finite element shape functions so that they satisfy certain jump conditions along the interface.

For some XFEM type of works related to the interface elasticity problems, we refer to
\cite{Belytschko_Parimi}, \cite{Belytschko_Black}, \cite{Krysl-Bely},  \cite{Legrain}, \cite{Moes:99}, where they
added enriched
basis functions obtained by multiplying Heaviside functions along the crack, and asymptotic basis of polar form near the tip. Even so, they often use grid refinement near the interface.
See Hansbo et al. \cite{BeckerB.Hansbo},\cite{A.P.Hansbo2002},\cite{A.P.Hansbo2004}, where they used
Nitsche's \cite{Nitsche}  idea of adding penalty terms along the
  interface of elements. 
For methods based on finite difference, see \cite{Lai_Li_L},\cite{LeVequeLi},\cite{Li1994},\cite{Oeverm_S_K}, for example.

 In this paper, we develop a new method based on the IFEM using the broken CR element for a linear
 elasticity problem having an interface.
 We modify the (vector) basis functions to satisfy the Laplace-Young condition along the interface.
 Our method does not use any extra shape function as in XFEM, hence our method yields exactly the same matrix structure as
the problems of constant Lam\'e parameters, and  has less degrees of freedom than XFEM.
Furthermore, numerical results show that our method does not need a mesh refinement.


 Near the completion of our first version of this manuscript \cite{Kwak-Jin}, we found that Lin et al. \cite{Lin-Zhang-Sh} have developed an IFEM similar to ours based on the rotated $Q_1$ nonconforming element without using stability terms to solve elasticity equations with interface, but  no analysis is given.
In contrast, we prove optimal error estimates of our scheme (based on CR $P_1$ nonconforming element), by adding stabilizing terms along the edges of elements for the coercivity of the bilinear forms.
 The rest of our paper is organized as follows. In section 2, we introduce the linear elasticity problems having interior interface along which the Laplace-Young condition holds and state their local regularity. For simplicity, we assume
 the Dirichlet data even though traction boundary condition on some part of boundary can be assigned.
  In section 3, we introduce our  new scheme for solving such problems using the CR $P_1$ nonconforming finite element.
  For this purpose, we modify the vector basis functions so that they satisfy  the Laplace-Young condition along the interface. In section 4, we introduce various norms and function spaces related to interface problems. Next we prove the approximation property of our finite element space and optimal error estimates in $H^1$ and divergence norm. Finally, numerical experiments are presented in section 5, which supports our results.

\section{Preliminaries}

Let $\Omega$ be a connected, convex polygonal domain in $\RR^2$ which is divided into two subdomains $\Omega^+$ and $\Omega^-$ by a $C^2$ interface $\Gamma = \partial \Omega^+ \cap \partial\Omega^- $, see Figure  \ref{fig:doma0}. We assume the subdomains $\Omega^+$ and $\Omega^-$ are occupied by two  elastic materials having
different Lam\'e constants.
 For a   differentiable function $\bv=(v_1,v_2) $ and a tensor $\btau=
\begin{pmatrix}  \tau_{11} &  \tau_{12} \\  \tau_{21}&  \tau_{22} \end{pmatrix} $, we let
 $$ \begin{array}{rl} \dstyle
  \nabla \bv=\begin{pmatrix} \pd{v_1}{x}& \pd{v_1}{y} \\ \pd{v_2}{x}& \pd{v_2}{y} \end{pmatrix}, \quad &
 \Div \btau = \begin{pmatrix} \pd{\tau_{11}}x +\pd{\tau_{12}}y \\ \pd{\tau_{21}}x +\pd{\tau_{22}}y  \end{pmatrix}.
 \end{array}$$
Then  the displacement $\bu=(u_1,u_2)$ of the elastic body under an external force satisfies the Navier-Lam\'e equation as follows.
\begin{eqnarray}
-\Div \bsigma(\bu)  &=& \mathbf{f} ~~ \mathrm{ in}~  \Omega^s, \, (s=+,-) \label{TheprimalEq} \\
{[\bu]}_\Gamma &=& 0,  \\
{[\bsigma(\bu)\cdot \bn]}_\Gamma&=&0, \label{jump2.3} \\
\bu &=& 0  ~~ \mathrm{ on}~  \partial\Omega, \label{jump2.4}
\end{eqnarray}
where \begin{equation}
 \bsigma(\bu) = 2\mu \bepsilon(\bu) + \lambda tr(\bepsilon(\bu))\boldsymbol{\delta},\,\, \bepsilon(\bu) = \frac{1}{2}(\nabla \bu + {\nabla \bu}^T)
\end{equation} are   the stress tensor  and the strain tensor respectively,
 $\bn$ is outward unit normal vector,  $\boldsymbol{\delta}$ is the identity tensor, and $\mathbf{f} \in (L^2(\Omega))^2$ is the external force. Here
  $$ \lambda=\frac{E\nu}{(1+\nu)(1-2\nu)},\quad \mu =\frac{E}{2(1+\nu) }$$
  are the Lam\'e constants  satisfying $0<\mu_1 <\mu <\mu_2$ and $0<\lambda<\infty$, $E$ is the Young's modulus, and $\nu$ is the Poisson ratio. When the parameter $\lambda \rightarrow \infty$, this equation describes the behavior of nearly incompressible material. Since the material properties are different in each region,  we set the Lam\'e constants $\mu = \mu^s, \lambda = \lambda^s ~~\mathrm{ on}~~ \Omega^s$ for $s = +,-$.
 The bracket $[\cdot]$ means the jump across the interface
 $$[\bu]_\Gamma := \bu|_{\Omega^+} - \bu|_{\Omega^-}.$$

 Multiplying $\bv \in (H^1_0 (\Omega))^2$ and applying Green's identity in each domain $\Omega^s$, we obtain
\begin{equation} \label{variational_form0}
\int_{\Omega^s} 2 \mu^s \bepsilon(\bu):\bepsilon(\bv)dx + \int_{\Omega^s} \lambda^s \Div \bu\,\Div \bv\,dx - \int_{\partial\Omega^s}\bsigma(\bu) \mathbf{n} \cdot \bv ds = \int_{\Omega^s} \mathbf{f} \cdot \bv dx,
\end{equation}
where 
\begin{equation*}
\bepsilon(\bu):\bepsilon(\bv) = \sum_{i,j=1}^2 \bepsilon_{ij}(\bu) \bepsilon_{ij} (\bv).
\end{equation*}
Summing over $s=+,-$ and applying the interior traction condition (\ref{jump2.3}), we obtain the following weak form
\begin{equation} \label{variational_form}
a(\bu,\bv) =  (\bff,\bv),
\end{equation}
where
\begin{equation} \label{a_form}
a(\bu,\bv) = \int_{\Omega} 2 \mu \bepsilon(\bu):\bepsilon(\bv)dx + \int_{\Omega} \lambda \Div \bu\,\Div \bv\,dx.
\end{equation} As usual, $(\cdot,\cdot)$ denotes the $L^2(\OM)$ inner product.
  Then we have the following result \cite{A.P.Hansbo2004},  \cite{Leguillon}.
\begin{thrm}
There exists a unique solution  $\bu \in (H_0^1(\Omega))^2$ of (\ref{TheprimalEq}) - (\ref{jump2.4}) satisfying
 and $\bu \in (H^2(\Omega^s))^2, s= +, -.$
Here, $H^1(\Omega), {H}^2(\Omega^s)$ etc., are usual Sobolev spaces on respective domains and $H_0^1(\Omega)$ is
a subspace of $H^1(\Omega)$ functions having zero trace.
\end{thrm}
 \begin{figure}[t]
\begin{center}
      \psset{unit=1.5cm}
      \begin{pspicture}(-1,-1)(1,1)
        \pspolygon(0.9,0.9)(-0.9,0.9)(-0.9,-0.9)(0.9,-0.9)
        \psccurve(0.47,0) (0.2,0.22)(-0.6,-0.1)(-0.2,-0.5)(0.5,-0.11)
        \rput(0,0){\scriptsize$\Omega^-$}
        \rput(-0.3,0.5){\scriptsize$\Omega^+$}
        \rput(0.67,0){\scriptsize$\Gamma$}

      \end{pspicture}
      \begin{pspicture}(-2.3,-1)(-0.3,1)
        \pspolygon(0.9,0.9)(-0.9,0.9)(-0.9,-0.9)(0.9,-0.9)

        \pscurve(-0.9,0.3) (-0.6,0.12)(0.1,-0.23)(0.5,0.25)(0.9,0.4)
        \rput(-0.23,-0.43){\scriptsize$\Omega^-$}
        \rput(-0.1,0.35){\scriptsize$\Omega^+$}
        \rput(0.5,0.1){\scriptsize$\Gamma$}

      \end{pspicture}

\caption{Domains $\Omega$ with interface}
  \label{fig:doma0}
\end{center}
\end{figure}
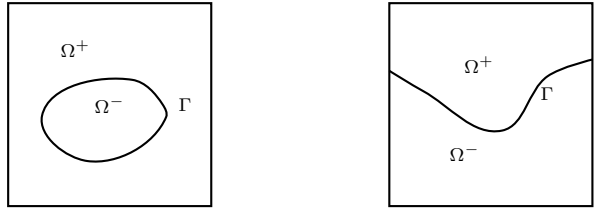
\section{An IFEM based on Crouzeix-Raviart element for the elasticity equation with interface}

In this section, we extend the CR type IFEM, which was first suggested by the author\cite{Kwak-We-Ch} for the elliptic problems to the elasticity equation with interface.
Before developing the scheme, we briefly review the stabilized version of FEM for the elasticity equation without interface (i.e., $\lambda^+= \lambda^-$
and $\mu^+=\mu^-$) introduced by Hansbo and Larson \cite{P.Hansbo_Lar2002}.

 Let $\{\mathcal{T}_h\}$ be a given quasi-uniform triangulations of $\Omega$ by the triangles of
  maximum diameter $h$. 
For each $T \in \mathcal{T}_h$, one constructs local basis functions using the average value along each edge as degrees of freedom. Let $$\overline{v}|_e = \frac 1 {|e|} \int_e v ds$$
denote the average of a function $v \in H^1(T)$ along an edge $e$ of $T$. Here  $|S|$  means the Lebesgue measure for any set $S\subset \RR^n, n=1,2,3$. Let $\bN_h(T)$ denote the linear space spanned by the six Lagrange basis functions
$$\bphi_i = (\phi_{i1},\phi_{i2})^T,\ i=1,2,3,4,5,6 $$
satisfying
$$\begin{aligned}
\overline{\phi_{i1}}|_{e_j} & = \delta_{ij},  & \overline{\phi_{i2}}|_{e_j} & =  \delta_{i-3,j},& j=1,2,3, &
\end{aligned}$$
where $\delta$ is the Kronecker delta.
The vector form of Crouzeix-Raviart $P_1$-nonconforming space is given by
$$ \bN_h(\Omega)= \left\{
\begin{aligned}
 \bphi: \bphi|_T\in \bN_h(T) &\mbox{ for each } T\in\cT_h;\ \mbox{ if $T_1$ and $T_2$ share an edge $e$,}\\
\mbox{ then } &\int_{e}{\bphi}|_{\partial T_1} ds= \int_{e}{\bphi}|_{\partial T_2}  ds; \mbox{ and }
  \int_{\partial T \cap \partial\Omega}{\bphi}\,ds=\b0
\end{aligned}
\right\}.$$

The stabilized $P_1$-nonconforming finite element method for (\ref{variational_form}) is : find $\bu_h \in \bN_h(\Omega)$ such that
\begin{equation}\label{Hansbo:disc-form}
a_h(\bu_h,\bv_h) = (\bff,\bv_h), ~\quad ~~ \forall \bv_h \in \bN_h (\Omega),
\end{equation}
where
\begin{eqnarray}\label{a_h_form}
a_h(\bu_h,\bv_h) :& =&\sum_{T\in \mathcal{T}_h}\int_T 2 \mu \bepsilon(\bu_h):\bepsilon(\bv_h)dx +
 \sum_{T\in \mathcal{T}_h}\int_T  \lambda \Div \bu_h\,\Div \bv_h\,dx  \nonumber\\
&& + \tau \sum_{e \in \mathcal{E}}  \int_{e} h^{-1} [\bu_h][\bv_h]ds \text{ for some $\tau > 0$}.
\end{eqnarray}
 For a problem without an interface,
Hansbo and Larson \cite{P.Hansbo_Lar2002} proved the following result.
 \begin{thrm} \label{Hansbo-error}
Let $\bu$ be the solution of (\ref{TheprimalEq}) and $\bu_h$ be the solution of (\ref{Hansbo:disc-form}). Then
$$\|\bu - \bu_h\|_{a_h} \leq Ch\|f\|_{L_2(\Omega)},$$
where $\|\cdot\|_{a_h}= a_h(\cdot,\cdot)^{1/2}$.
\end{thrm}

\subsection*{Construction of broken CR-basis functions satisfying Laplace-Young condition}

Now we are ready to introduce our IFEM. We consider an elasticity equation with an interface.
Let $\{\mathcal{T}_h\}$ be any quasi-uniform triangulations of $\Omega$ of maximum diameter $h$. We allow the grid to be cut by the interface. 

\begin{figure}[ht]
  \begin{center}
    \psset{unit=2.8cm}
    \begin{pspicture}(0,0)(1,1)
      \psset{linecolor=black} \pspolygon(0,0)(1,0)(0,1) \psline(0,0.65)(0.35,0)
      \pscurve(0,0.65)(0.1,0.58)(0.2,0.15)(0.35,0)
      \rput(0,1.05){\scriptsize$A_3$}
      \rput(-0.05,0){\scriptsize$A_1$}
      \rput(1.05,0){\scriptsize$A_2$}
      \rput(0.725,-0.07){$e_3$}
      \rput(0.55,0.55){$e_1$}
      \rput(-0.06,0.5){$e_2$}
      \rput(-0.25,0.65){\scriptsize$E=(0,y)$}
      \rput(0.08,0.12){\scriptsize$T^-$}
      \rput(0.45,0.25){\scriptsize$T^+$}
      \rput(0.35,-0.08){\scriptsize$D=(x,0)$}
      \rput(0.20,0.48){\scriptsize$\Gamma$}
\pnode(-.3,0.6){a}
\pnode(0.12,0.5){b}
%
    \end{pspicture}
    \caption{A typical interface triangle} \label{fig:interel}
\end{center}
\end{figure}
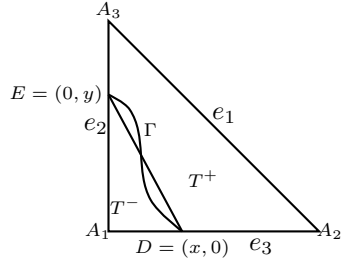
We call an element $T\in\mathcal{T}_h$ an {\it interface element} if the interface $\Gamma$ passes through the interior of $T$, otherwise we call it a {\it noninterface} element.  Let $\mathcal{T}^*_h$ be the collection of all interface elements. We assume the following situations which are easily satisfied  when $h$ is small enough:
\begin{itemize}
\item the interface intersects the edges of an element at no more than two points. 
\item the interface intersects each edge at most once, except possibly it passes through two vertices.
\end{itemize}

The main idea of the IFEM for elasticity problem is to use two pieces of linear shape functions (vector form) on an interface  element to   satisfy the Laplace-Young condition. We set, for $i=1,2,\cdots,6$,
\begin{eqnarray}
 &&\hat\bphi_i(x,y) = \left\{%
 \begin{array}{ll}
    \hat\bphi^+_i (x,y) =
    \begin{pmatrix}
     \hat\phi^+_{i1} \\  \hat\phi^+_{i2}
    \end{pmatrix} = \begin{pmatrix} a_1^+ + b_1^+ x + c_1^+ y \\
    a_2^+ + b_2^+ x + c_2^+ y \end{pmatrix} , \quad (x,y) \in  T^+ \\
     &\\
    \hat\bphi^-_i (x,y) =\begin{pmatrix}
    \hat \phi^-_{i1} \\  \hat\phi^-_{i2}
    \end{pmatrix} = \begin{pmatrix} a_1^- + b_1^- x + c_1^- y  \\
    a_2^- + b_2^- x + c_2^- y \end{pmatrix} , \quad  (x,y) \in T^-\\
\end{array}%
\right. \label{def:basis-1}
\end{eqnarray} and require these functions satisfy the $6$ degrees of freedom (edge average), continuity, and jump conditions:
\begin{equation} \label{eq:dof1}
\begin{array}{ccl}
\overline{\hat\phi_{i1}}|e_j & = & \delta_{ij},\, j=1,2,3\\
\overline{\hat\phi_{i2}}|e_j & = & \delta_{(i-3)j},\, j=1,2,3 \\
{[\hat \bphi_i(D)]} & = & 0, \\
{[\hat \bphi_i(E)]} & = & 0, \\
\dstyle\left[\bsigma(\hat\bphi_i) \cdot \mathbf{n}\right]_{\overline{\textrm{\tiny{$DE$}}}} & = & 0 .
\end{array}
\end{equation}
These twelve conditions lead to a system of linear equations in twelve unknowns for each $i$.

\begin{proposition}\label{prop:exiten-CRbasis}
The conditions  (\ref{eq:dof1}) uniquely determine the function $\hat\bphi_i$
 of the form  (\ref{def:basis-1}), regardless of the interface locations.
\end{proposition}
\begin{proof}
See Appendix A for details.
\end{proof}

We denote by $\what{\bN}_h(T)$ the space of functions generated by $\hat{\bphi}_i, i = 1,2,3,4,5,6$ constructed above. Using this local finite element space, we define the global {\em immersed finite element space} $\what{\bN}_h (\Omega)$
by
$$ \what{\bN}_h (\Omega)=
\left\{\begin{array}{l}
\hat \bphi\in \what{\bN}_h(T) \mbox{ if } T\in \mathcal{T}_h^*, \mbox{ and }  \hat\bphi\in \bN_h(T) \mbox{ if } T \not\in \mathcal{T}_h^*; \\
\mbox{ if $T_1$ and $T_2$ share an edge $e$, then }\\
  \int_{e}{\hat\bphi}|_{\partial T_1} ds= \int_{e}{\hat\bphi}|_{\partial T_2}  ds; \mbox{ and }
  \int_{\partial T \cap \partial\Omega}{\hat\bphi}\,ds=\b0
 \end{array}
\right\}. $$

We now propose an IFEM scheme for (\ref{TheprimalEq})-(\ref{jump2.4}).
\subsection*{CRIFEM}
Find $\bu_h \in \what{\bN}_h(\Omega) $ such that
\begin{equation} \label{Discrete}
a_h(\bu_h,\bv_h) =  (\bff,\bv_h),\, ~~~ \forall \bv_h \in \what{\bN}_h(\Omega),
\end{equation} where
$a_h(\cdot,\cdot)$ is the same as  (\ref{a_h_form}).

\section{Error Analysis}
  We introduce function spaces and norms that are necessary for analysis.
Let   $p\geq1$ and  $m\geq 0$ be an integer. For any domain $D$, we let $W_p^m(D)$ ($H^m(D)=W_2^m(D)$) be the usual Sobolev space with (semi)-norms denoted by $|\cdot|_{m,p,D}$ and $\|\cdot\|_{m,p,D}$. ($\|\cdot\|_{m,D}=\|\cdot\|_{m,2,D}$).
  For $m=1,2$ and any domain $D=T (\in \mathcal{T}_h)$ or $D=\Omega$, let
\begin{eqnarray*}
(\wtilde{W}^m_p(D))^2:=\{\bu \, = (u_1, u_2) \in (W^{m-1}_p(D))^2 : \,u|_{D\cap \Omega^s}\in
(W^m_p(D\cap \Omega^s))^2, s = +,-\,\}~~~
\end{eqnarray*}
 with norms
\begin{eqnarray*}
|\bu|^p_{\wtilde{W}^m(D)} := |\bu|^p_{m,p,D\cap \Omega^+} + |\bu|^p_{m,p,D\cap \Omega^-},\,\,\mbox{ and }
\|\bu\|^p_{\wtilde{W}^m(D)} := \|u\|^p_{m-1,p,D} + |\bu|^p_{\wtilde{W}^m(D)}.
\end{eqnarray*}
When $p=2$, we write  $(\wtilde{H}^m(D))^2:=(\wtilde{W}^m_p(D))^2$ and denote
the  norms (resp. semi norms) by $\|\bu\|_{\wtilde{H}^m(D)}$ (resp. $|\bu|_{\wtilde{H}^m(D)}$),  etc.
  When a finite element
triangulation $\{\mathcal{T}_h\}$ is involved, the norms are understood as  piecewise norms $(\sum_{T\in\mathcal{T}_h}\|\bu\|^p_{\wtilde{W}^m_p(T)})^{1/p}$, etc. If $p=2$, we denote them by  $\|\bu\|_{m,h}$ (resp. $|\bu|_{m,h}$).
Let $\bH_h(\Omega) := (H_0^1(\Omega))^2 + \what{\bN}_h (\Omega)$.
 We need subspaces of $(\wtilde{H}^2(T))^2$ and $(\wtilde{H}^2(\Omega))^2$  satisfying the jump conditions:
\begin{eqnarray*}
(\wtilde{H}^2_{\Gamma}(T))^2 \hspace{-4pt}  &:=& \hspace{-4pt} \{\,\bu \in (\wtilde{H}^2(T))^2 \, \mbox{ and } \left[\bsigma(\bu) \cdot \bn \right]_{\Gamma\cap T} = 0 \},\\
(\wtilde{H}^2_{\Gamma}(\Omega))^2 \hspace{-4pt}  &:=& \hspace{-4pt} \{ \,\bu\in ( H^1_0(\Omega))^2 \, \bu|_{T} \in (\wtilde{H}^2_{\Gamma}(T))^2, \,\, \forall T\in \mathcal{T}_h \}.
\end{eqnarray*}

 Throughout the paper, the constants $C, C_0, C_1$, etc., are generic constants  independent
of the mesh size $h$ and functions $\bu, \bv$ but may depend on the problem data  $\mu, \lambda, \bff$ and $\Omega$, and are not necessarily the same on each occurrence.
\subsection{Approximation property of  $\what{\bN}_h(T)$}

Note that the case of a scalar elliptic problem is given in \cite{Kwak-We-Ch}.
One of the obstacles in proving the approximation property is: the space $\what{\bN}_h(T)$ does not belong to $ (\wtilde{H}^2_{\Gamma}(T))^2$ because the curved interface is approximated by the line segement.
 To overcome this difficulty, we introduce a bigger space which contains both of these spaces.
For a given interface element $T$, we define  function spaces $X(T)$ and $X_\Gamma(T)$ by
\begin{eqnarray}\label{sp:XT}
X(T)&:=& \left\{\bu :\bu \in (H^1(T))^2, \bu \in (H^2(S))^2 \mbox{ for all } S=T^+_r,T^-_r, T^+\cap\Omega^+, T^-\cap\Omega^-\right\}\\
 X_\Gamma(T)&:=& \left\{\bu :
\bu \in X(T),\int_{\Gamma\cap T}(\bsigma(\bu)^- - \bsigma(\bu)^+)\cdot \mathbf{n}_\Gamma\, ds = 0\right\}
\end{eqnarray}
where $\bsigma(\bu)^- = 2\mu^- \bepsilon(\bu) + \lambda^- \Div \bu$, $\bsigma(\bu)^+ = 2\mu^+ \bepsilon(\bu) + \lambda^+ \Div \bu$ and
$S=T^+_r,T^-_r, T^+\cap\Omega^+, T^-\cap\Omega^- $ are subregions of $T$ created by $\Gamma$ and line segment $\overline{\textrm{$DE$}}$ (See Fig. \ref{fig:regul}).

Note the relations
\begin{eqnarray}\label{sp:relation}
 (\wtilde{H}^2(T))^2 \hookrightarrow X(T) \hookrightarrow ({H}^1(T))^2 \\
(\wtilde{H}_\Gamma^2(T))^2\cup\what{\bN}_h(T)\hookrightarrow X_\Gamma(T)\hookrightarrow X(T)\hookrightarrow ({H}^1(T))^2
\end{eqnarray}
  For any $\bu \in  X(T) $, we define the
following norms:
\begin{eqnarray*}
\|\bu\|_{b,m,T}^2 &=&  \|\bu\|^2_{m,T}+m\cdot\|\sqrt{\lambda}\Div\bu\|^2_{0,T} , \, m=0,1\\
|\bu|^2_{X(T)} &=& |\bu|^2_{2,T^-\cap\Omega^-} + |\bu|^2_{2,T^+\cap\Omega^+} +
|\bu|^2_{2,T^-_r} + |\bu|^2_{2,T^+_r}, \\
\|\bu\|^2_{X(T)} &=& \|\bu\|^2_{1,T}+|\bu|^2_{X(T)} +\|\sqrt{\lambda}\Div\bu\|^2_{0,T}+ \sum_{s=+,-}|\sqrt{\lambda}\Div\bu|^2_{1,T^s} \\
\tnorm{\bu}^2_{2,T}&=&|\bu|^2_{X(T)} + \sum_{s=+,-}|\sqrt{\lambda}\Div\bu|^2_{1,T^s} \\
 &+&\left|\int_{\Gamma\cap T}[\bsigma(\bu) \mathbf{n}_\Gamma]\, ds\right|^2
 + \sum^3_{i=1}|\overline{u_1}|_{e_i}|^2+\sum^3_{i=1}|\overline{u_2}|_{e_i}|^2,
\end{eqnarray*}
Note that when $m=0$, $\|\bu\|^2_{b,m,T}$ is just the $L^2$-norm  $\|\bu\|^2_{0,T}$.
For $\bv\in \bH_h(\Omega)$, define
\begin{eqnarray} \label{Energy_norm}
\|\bv\|^2_{a_h}&:=&a_h(\bv,\bv)=\sum_{T\in\mathcal{T}_h}\left(\int_{T} 2\mu \epsilon(\bv):\epsilon(\bv) dx +\int_{T} \lambda |\Div \bv|^2 dx
\right)+ \sum_{e\in \mathcal{E}}\int_{e} \frac{\tau}{h}[\bv]^2 ds.
\end{eqnarray}

\begin{figure}[ht]
  \begin{center}
    \psset{unit=2.7cm}
    \begin{pspicture}(-1.5,-0.17)(1.5,1)
    \newrgbcolor{mycolor}{0.3 0.2 0.0}  \newcmykcolor{mycolor2}{0.4 0.0 0.1 0}
      \psset{linecolor=black} \pspolygon(-1.25,0)(-0.25,0)(-1.25,1)
      \psline(-1.25,0.65)(-0.9,0)
      \rput(-1.3,0){\scriptsize$A$}
      \rput(-0.2,0){\scriptsize$B$}
      \rput(-1.25,1.05){\scriptsize$C$}
      \rput(-1.12,0.12){\scriptsize$T^-$}
      \rput(-0.8,0.25){\scriptsize$T^+$}
      \rput(-0.9,-0.05){\scriptsize$D$}
      \rput(-1.3,0.65){\scriptsize$E$}
      \psset{linecolor=black}
      \pspolygon(0.25,0)(1.25,0)(0.25,1)
\pscustom[fillstyle=solid,fillcolor=lightgray]{
  \psline(0.6,0) (1.25,0)(0.25,1)(0.25,0.65)
 \pscurve(0.25,0.65)(0.35,0.58)(0.45,0.15)(0.6,0)
        }
\pscustom[fillstyle=solid,fillcolor=mycolor2]{
 \pscurve(0.25,0.65)(0.35,0.58)(0.45,0.15)(0.6,0)\psline(0.6,0) (0.25,0)(0.25,0.65)
        }
     \psline(.25,0.65)(0.6,0)
 \pscustom[fillstyle=vlines,fillcolor=gray,hatchsep=3pt]{
 \pscurve(0.25,0.65)(0.35,0.58)(0.45,0.15)(0.6,0)\psline(0.6,0) (.25,0.65)
  }
      \rput(0.2,0){\scriptsize$A$}
      \rput(1.3,0){\scriptsize$B$}
      \rput(0.25,1.05){\scriptsize$C$}
      \rput(0.365,0.1){\scriptsize$T\cap\Omega^-$}
      \rput(0.7,0.25){\scriptsize$T\cap\Omega^+$}
      \rput(0.6,-0.05){\scriptsize$D$}
      \rput(0.2,0.65){\scriptsize$E$}
      \rput(0.45,0.48){\scriptsize$\Gamma$}
 	 \rput(0.7,-0.08){\scriptsize$e_3$}
      \rput(0.78,0.53){\scriptsize$e_1$}
	 \rput(0.17,0.46){\scriptsize$e_2$}
      \rput(-0.8,-0.16){\scriptsize{(a) $\what{N}_h(T)\subset H^2(T^+)\cap H^2(T^-)$}}
      \rput(0.8,-0.16){\scriptsize{(b) $\wtilde{H}^2(T)\subset H^2(T\cap\Omega^+)\cap H^2(T\cap\Omega^-)$}}
\rput(0.22,0){
\rput(-.35,0.65){\scriptsize$T_r^+$}
\pnode(-.3,0.6){a}
\pnode(0.12,0.5){b}
\ncarc[linewidth=0.5\pslinewidth]{->}{a}{b}
\pnode(-.3,.3){a}
\pnode(0.22,0.21){b}
\ncarc[linewidth=0.5\pslinewidth]{->}{a}{b}
\rput(-.35,0.23){\scriptsize$T_r^-$}
}
    \end{pspicture}
    \caption{The real interface and the approximated interface} \label{fig:regul}
\end{center}
\end{figure}
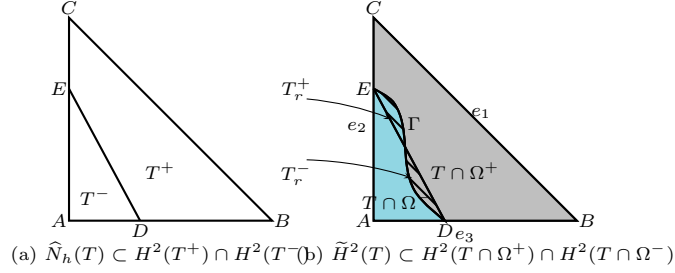

\begin{remark}\label{remark:X(T)}
\begin{enumerate}
\item The difference between the spaces $(\wtilde{H}^2(T))^2$ and $X(T)$ is this : a function
$\bu\in (\wtilde{H}^2(T))^2$ is ${H}^2$ in each of the regions $T^+$ and $T^-$ while a function
$\bu\in X(T)$ is ${H}^2$ in each of the four regions $T^+_r,T^-_r, T^+\cap\Omega^+, T^-\cap\Omega^-$.
\item   The difference between the spaces $(\wtilde{H}_\Gamma^2(T))^2$ and $X_\Gamma(T)$ is this : a function
$\bu\in(\wtilde{H}_\Gamma^2(T))^2$ satisfies the  a strong Laplace-Young condition (\ref{jump2.3}) along $\Gamma$  while  $\bu \in X_\Gamma(T)$ satisfies it weakly: $\int_{\Gamma\cap T}(\bsigma(\bu)^- - \bsigma(\bu)^+)\cdot \mathbf{n}_\Gamma\, ds = 0.$ In fact, for  every continuous, piecewise linear functions, this condition holds  if and only if it satisfies the same condition along the line segment joining the end points of the interface, as shown in the Lemma below.
\end{enumerate}
\end{remark}

\begin{lmm} \label{thm:wflux}
For an interface triangle $T$,  every continuous, piecewise linear function $\bphi$ satisfies
\begin{eqnarray}\label{weak:curv-cont}
\int_{\Gamma\cap T} [\bsigma(\bphi)\cdot\mathbf{n}_\Gamma] ds = 0 \mbox{ if and only if } \int_{\overline{DE}} [\bsigma(\bphi)\cdot\mathbf{n}_{\overline{\textrm{\tiny{$DE$}}}}] ds = 0.
\end{eqnarray}
\end{lmm}
\begin{proof}
 This can be easily proved by Green's theorem since  $\bphi$ is piecewise linear.
\end{proof}

\begin{lmm} \label{lem:norm}
$\tnorm{\cdot}_{2,T}$ is a norm on the space $X_{\Gamma}(T)$ which is equivalent to
$\|\cdot\|_{X(T)}$.
\end{lmm}
\begin{proof}
Clearly,  $\tnorm{\cdot}_{2,T}$ is a semi-norm. To show it is indeed a norm, assume
  $\bu \in X_{\Gamma}(T) $ satisfies $\tnorm{\bu}_{2,T} = 0$. Then $|\bu|_{X(T)} = 0$. Hence  $\bu$ is linear on each of the
four regions $T^+\cap\Omega^+$, $T^-\cap\Omega^-$, $T^+_r$ and $T^-_r$.
 Since $\bu \in H^1(T)$, $\bu$ is continuous on  $T$. Since $\int_{\Gamma\cap T}[\bsigma(\bu)\cdot \mathbf{n}_\Gamma]\, ds=0,$
$\bu$ satisfies the interface condition along the line segment $\overline{DE}$  by Lemma \ref{thm:wflux}. Hence $\bu \in \what{\bN}_h(T)$ and
together with the fact that  $\overline{u_1}_{e_i}=0,~ i=1,2,3$ and $\overline{u_2}_{e_i}=0,~ i=1,2,3$, we conclude $\bu = 0$, which shows that $\tnorm{\cdot}_{2,T}$ is a norm.

We now show the equivalence of $\tnorm{\cdot}_{2,T}$ and $\|\cdot\|_{X(T)}$ on the space $ X_{\Gamma}(T)$.
(cf. \cite[p.77]{Braess}). By Sobolev embedding,
\begin{eqnarray}
\sum^3_{i=1} | \overline{u_1}_{e_i}| + \sum^3_{i=1} | \overline{u_2}_{e_i}| &\leq& C\max_{s=+,-}\|\bu\|_{L^{\infty}({T^s})} \\
&\leq& C\max_{s=+,-}\|\bu\|_{H^2(T^s)} \leq C\|\bu\|_{\wtilde{H}^2(T)}. \\
&\leq& C\|\bu\|_{X(T)} . \label{eq:Traces2}
\end{eqnarray}
 Hence we see
\begin{eqnarray}
\tnorm{\bu}_{2,T}&\leq& C\|\bu\|_{X(T)}. \label{eq:norms}
\end{eqnarray}

Now suppose that the converse
\begin{eqnarray*}
\|\bu\|_{X(T)}\leq C\tnorm{\bu}_{2,T},~~\forall \bu\in X_{\Gamma}(T)
\end{eqnarray*}
fails for any $C>0$. Then there exists a sequence $\{\bu_k\}$ in $X_{\Gamma}(T)$ with
\begin{eqnarray}\label{eq:Kondra1}
\|\bu_k\|_{X(T)}=1,~~\tnorm{\bu_k}_{2,T}\leq \frac{1}{k},~~k=1,2,\cdots .
\end{eqnarray}
Let $S_t, t=1,\cdots,4$ denote the four subregions defined in the definition of $X(T)$. Since $H^2(S_1)$ is compactly embedded in $H^1(S_1)$, \cite[p.114]{Ciarlet}, there exists a subsequence  of $\{\bu_k\}$
which converges in $(H^1(S_1))^2$. Applying the same argument successively  to the subsequences of previous ones on
 $S_2,S_3, S_4$,  we can  choose a subsequence, call $\{\bu_k\}$ again,
which converges on each of $S_t, t=1,2,3,4$.  Call its limit $\bu^*= (u^*,v^*)$. We claim that
$\bu^*\in (H^1(T))^2$. Note that $T=\cup_{t=1}^4 S_t.$ 
For simplicity, we assume the interface is a line segment so that $T= T^+\cup T^-$. The same argument shows
 that $\bu^*\in (H^1(T))^2$ when $T$ consists of four pieces, $T=\cup_{t=1}^4 S_t.$
Let us denote $\bu_k = (u_k,v_k)$ and $\bu^* = (u^*,v^*)$ respectively, and $u_k^s= u_k|_{T^s}, u^s = u^*|_{T^s}, s= +,-$.
Let $n^s_1$ be the first component of the unit outer normal vector to the boundary of $T^s, s=+,-$.

  By Green's theorem, and the fact $\lim_{k \rightarrow \infty} u^+_k= u^+=u^-$ on $\Gamma$, we get for $\phi\in C_0^\infty(T)$
\begin{eqnarray*}
\int_{T^+} \pd{u^+}{x}\phi dx&=& \left(\int_{\partial T^+} u^+ n^+_1 \phi\, ds- \int_{T^+} u^+ \pd{\phi}{x}dx \right)\\
&=& \lim_{k \rightarrow \infty} \left(\int_{\Gamma}u^-_k n^+_1 \phi \,ds -\int_{T^+} u^+_k \pd{ \phi}{x} dx\right).
\end{eqnarray*}
Similarly,
\begin{eqnarray*}
\int_{T^-} \pd{u^-}{x}\phi dx&=& \left( \int_{\partial T^-} u^- n^-_1 \phi\, ds- \int_{T^-} u^-\pd{\phi}{x}dx \right)\\
  &=& \lim_{k \rightarrow \infty}\left( \int_{\Gamma} u^-_k  n^-_1 \phi\, ds- \int_{T^-} u^-_k \pd{\phi}{x}dx \right).
\end{eqnarray*}   Adding these two equations, we have
\begin{eqnarray*}
\int_{T^+} \pd{u^+}{x}\phi dx +  \int_{T^-} \pd{u^-}{x}\phi dx&=&  -\int_{T} u^*\pd{\phi}{x}dx.
\end{eqnarray*}
So if we define the function $w$ by
$$ w = \begin{cases} \pd{u^+_i}{x} &\mbox{ on }  T^+\\
 \pd{u^-_i}{x}  &\mbox{ on }  T^-
\end{cases}   $$
then it satisfies
$$ \int_{T} w \phi dx =- \int_{T} u^* \pd {\phi}{x} dx,\quad \phi\in C_0^\infty(T). $$
This shows $\pd {u^*}{x}$ is well defined in $L^2(T)$.
  The same argument shows that $\pd{u^*}{y}$ is also well defined in $L^2(T)$.
  The same argument applied to $v^*$ shows  $\bu^*=(u^*,v^*)\in (H^1(T))^2$ and hence $ \|\bu_k-\bu^*\|_{1,T}\to 0$.
Since
\begin{eqnarray*}
\|\bu_k -\bu_l\|^2_{X(T)} &=& \|\bu_k-\bu_l\|^2_{1,T}+\|\sqrt{\lambda}\Div(\bu_k-\bu_l)\|^2_{0,T}+ \tnorm{\bu_k-\bu_l}_{2,T}^2\\
&\leq& \|\bu_k-\bu^*\|^2_{1,T}+\|\bu^*-\bu_l\|^2_{1,T}\\
&& +\|\sqrt{\lambda}\Div(\bu_k-\bu^*)\|^2_{0,T}+\|\sqrt{\lambda}\Div(\bu^*-\bu_l)\|^2_{0,T}+(\tnorm{\bu_k}_{2,T}+\tnorm{\bu_l}_{2,T})^2\to 0
\end{eqnarray*} as $k,l\to\infty$,
we see that $\{\bu_k\}$ is a Cauchy sequence in $X_\Gamma(T)$. By completeness, it converges to a limit in $X_\Gamma(T)$ which is  $\bu^*$ and hence
\begin{eqnarray}\label{eq:normequiv1}
\|\bu^* \|_{X(T)} =\lim_{k\to\infty}\|\bu_k \|_{X(T)}= 1.
\end{eqnarray}
 Now (\ref{eq:norms}), (\ref{eq:Kondra1}) gives
 $$ \tnorm {\bu^*}_{2,T}\le \tnorm{\bu^*-\bu_k}_{2,T}+\tnorm{\bu_k}_{2,T} \le  C\|{\bu^*-\bu_k}\|_{X(T)}+\frac1k\to 0,$$
 this implies $\bu^*=0$. But this is a contradiction to (\ref{eq:normequiv1}).
\end{proof}

We define an interpolation operator:
for any $\bu \in (H^1(T))^2$, we define $I_h \bu \in \what{\bN}_h(T)$ using the average of
$\bu$ on each edge of $T$ 
by
$$ \int _{e_i} I_h \bu\, ds = \int _{e_i}   \bu\, ds,~~ i=1,2,3$$
and call $I_h \bu$ the \emph{interpolant} of $\bu$ in $\what{\bN}_h(T)$. We then
define $I_h \bu $ for $ \bu \in (H^1(\OM))^2 $ by $(I_h \bu)|_{T}= I_h (\bu|_T)$.

Now we are ready to prove  the interpolation error estimate.
\begin{proposition}\label{thm:apperror}
For any $\bu \in (\wtilde{H}^2_\Gamma(\Omega))^2$, there exists a constant $C>0$ such that for $ m=0,1$
\begin{eqnarray*}
 \|\bu - I_h \bu\|_{m,h}+m\cdot\|\sqrt{\lambda}\Div (\bu - I_h\bu)\|_{L^2(\Omega)} \leq C h^{2-m} (
\|\bu\|_{\wtilde{H}^2(\Omega)}+m\cdot\sqrt{\lambda_M}\|\Div \bu\|_{\wtilde{H}^1(\Omega)}), \,
\end{eqnarray*}
and
\begin{eqnarray*}
\|\bu-I_h \bu\|_{m,h}\leq Ch^{2-m}\|\bu\|_{\wtilde{H}^2(\Omega)}.
\end{eqnarray*}
 
\end{proposition}

\begin{proof}
 Let $\check{T}$ be a reference interface element, $\check{\Gamma}$ be the corresponding local reference interface, and $\check{\bu}(\check\bx):=\bu\circ \bF(\check\bx)$, where
 $ \bF:\check{T}\to T $ denote the affine mapping to define the finite element in the real domain. Then for any $\check{\bu}\in  (\wtilde{H}^2_\Gamma(\check{T}))^2\subset
X_\Gamma(\check{T})  $, (let us denote $\check{\bu} = (\check u_1,\check u_2)$ and $I_h \check{\bu} = (\check w_1,\check w_2)$)
\begin{eqnarray*}
 \tnorm{\check{\bu}-I_h\check{\bu}}^2_{2,\check{T}} &=&|\check{\bu}-I_h\check{\bu}|^2_{X(\check{T})}+
 \sum_{s=+,-}|\sqrt{\lambda}\Div(\check{\bu}-I_h\check{\bu})|^2_{1,\check{T}^s}\\
&&+ \left|\int_{\check{\Gamma}\cap \check{T}}[(\bsigma(\check\bu) -\bsigma(I_h \check\bu)) \cdot \mathbf{n}_\Gamma]\, ds\right|^2
 +\sum^3_{i=1} |(\overline{\check u_1-\check w_1})|_{e_i}|^2+\sum^3_{i=1}|(\overline{\check u_2-\check w_2})|_{e_i}|^2\\
&=& |\check{\bu} - I_h \check{\bu}|^2_{X(\check{T})} +\sum_{s=+,-}|\sqrt{\lambda}\Div(\check{\bu}-I_h\check{\bu})|^2_{1,\check{T}^s} =
|\check{\bu}|^2_{X(\check{T})} + \sum_{s=+,-}|\sqrt{\lambda}\Div\, \check {\bu}|^2_{1,\check{T}^s},
\end{eqnarray*} 
where we used the properties of the interpolation operator $I_h$, Lemma \ref{thm:wflux},
and the  fact that $H^2$-seminorm of the piecewise linear
function $I_h \tilde{\bu}$ vanishes.
 
Let $m=0$ or $1$. By Lemma \ref{lem:norm} and scaling argument,
\begin{eqnarray*}
 \|\bu- I_h\bu\|_{b,m,T} &\leq& Ch^{1-m}\|\check{\bu}-I_h\check{\bu}\|_{b,m,\check{T}}   \\
 &\leq& Ch^{1-m}\|\check{\bu}-I_h\check{\bu}\|_{X(\check{T})} \\
 &\leq& Ch^{1-m}\tnorm{\check{\bu}-I_h\check{\bu} }_{2,\check{T}}  \\
&=& Ch^{1-m}(|\check{\bu} |_{X(\check{T})}+m\cdot\sum_{s=+,-}|\sqrt{\lambda}\Div\,\check{\bu} |_{1,\check{T}^s} )\\
&\leq& Ch^{2-m}(|\bu|_{X(T)}+m\cdot\sum_{s=+,-}|\sqrt{\lambda}\Div\,\bu |_{1,T^s})\\
&\leq& Ch^{2-m}(\|\bu\|_{\wtilde{H}^2(T)}+m\cdot\sum_{s=+,-}|\sqrt{\lambda}\Div\,\bu |_{1,T^s}).
\end{eqnarray*}
For the second assertion
 one can proceed exactly the same way
 without the terms involving  $\Div\bu$ in the definition of norms $\|\cdot\|_{b,m,T}, \|\cdot\|_{X(T)}$
 and $\tnorm{\cdot}_{2,T}$ to obtain the desired estimate.
\end{proof}

\begin{proposition} \label{1-a-norm}
Let $\bu \in (\wtilde{H}^2_\Gamma(\Omega))^2$. We have
\begin{eqnarray}\label{1-a-norm-eq}
\|\bu-I_h\bu\|_{a_h} \leq Ch\left(\|\bu\|_{\wtilde{H}^2(\Omega)} + \sqrt{\lambda_M}\|\Div \bu\|_{\wtilde{H}^1(\Omega)}\right),
\end{eqnarray}
for some constant $C > 0$.
\end{proposition}
\begin{proof}
Recall that
\begin{eqnarray*} 
\|\bu-I_h\bu\|^2_{a_h}&=& \sum_{T\in\mathcal{T}_h}\int_{T} \left(2\mu \epsilon(\bu-I_h\bu):\epsilon(\bu-I_h\bu) + \lambda |\Div (\bu-I_h\bu)|^2\right) dx
 + \sum_{e\in \mathcal{E}}\int_{e} \frac{\tau}{h}[\bu-I_h\bu]^2 ds.
\end{eqnarray*}
Clearly, the terms in the first summation are bounded by the $\|\bu-I_h\bu\|_{b,1,T}$ for each element. Hence these are bounded by right hand side of (\ref{1-a-norm-eq}) by Proposition \ref{thm:apperror}.
For the second term, we have
\begin{eqnarray}
\frac{1}{h}\|[\bu-I_h\bu]\|^2_{0,e} &\leq&
\frac{1}{h}\|\bu-I_h\bu\|^2_{0,e} \\
&\leq& C\left(\frac{1}{h^2}\|\bu-I_h\bu\|^2_{0,T} + |\bu-I_h\bu|^2_{1,T}\right) \\
 &\leq& Ch^2\|\bu\|^2_{\wtilde{H}^2(T)},
\end{eqnarray} by trace inequality and Proposition \ref{thm:apperror}.
This completes the proof.
\end{proof}

\begin{lmm} [Korn's inequality  \cite{Brenner-Korn},\cite{Ciarlet-Elast}] \label{korn}
There exists constant $C >0$ such that
\begin{equation} \label{H1P-Korns3}
 \, |\bv_h|_{1,h}^2\le C\sum_{T\in\mathcal{T}_h} (\|\bepsilon(\bv_h)\|^2_{0,T}+ \|Q(\bv_h)\|^2_{0,T})+  \sum_{e \in \mathcal{E}}  \int_{e}\frac \tau h [\bv_h]^2ds , \forall \bv_h\in \what\bN(T),
\end{equation}
where $Q(\bv_h):=\bv_h-\frac{1}{|T|}\int_T \bv_h\,dx.$
 \end{lmm}
 \begin{corollary} \label{H1P-Korns2}
 The form $a_h(\cdot,\cdot)$ is a norm equivalent to $\|\cdot\|_{1,h}$.
\end{corollary}
\begin{proof} There exists a constant $C(T) >0$ such that the following holds.
  \begin{eqnarray*}
 \|Q(\bv_h)\|^2_{0,T} &\le&  C(T)h |\bv_h|^2_{1,T}.
\end{eqnarray*}
Hence by Lemma \ref{korn}, we have
  $$|\bv_h|^2_{1,h} \le C \sum_{T\in\mathcal{T}_h}\left(\|\bepsilon(\bv_h)\|^2_{0,T} +\int_{T} \lambda |\Div \bv_h|^2 dx\right)
 + \sum_{e \in \mathcal{E}} \int_{e} \frac \tau h [\bv_h]^2ds,\ \mbox{ for all } \bv_h\in \what\bN(T)  $$
holds for sufficiently small $h$.
 Hence by Poincar\'e inequality for CR finite element spaces \cite{Kwak-We-Ch}, we get the result.
\end{proof}

\begin{lmm} \label{consistency-term}
Let $\bu \in \tilde{H}^2(\Omega)$ be the solution of (\ref{TheprimalEq}).  We assume $\sigma(\bu)\cdot\bn\in (H^1(T))^2$ for each $T$. Then following inequality holds:
$$|a_h(\bu,\bv_h)-a_h(\bu_h,\bv_h)| \leq Ch R(\bu) \|\bv_h\|_{a_h},$$
where
$$ R(\bu)= \|\bu\|_{\wtilde{H}^2(\Omega)} +\lambda_M \|\Div \bu\|_{\wtilde{H}^1(\Omega)}. $$
\end{lmm}

\begin{proof}
Using the technique in \cite{Crouzeix}, we see that
the consistency error term  satisfies
\begin{eqnarray*}
\left|a_h(\bu,\bv_h)-a_h(\bu_h,\bv_h)\right| &=&
  \left|\sum_{T\in\mathcal{T}_h} \sum_{e \subset \partial T} \int_{e} \bsigma(\bu) \bn   \cdot [\bv_h] ds\right|\\
  &=&  \left|\sum_{T\in\mathcal{T}_h} \sum_{e \subset \partial T} \int_{e} (\bsigma(\bu) \bn-\overline{\bsigma(\bu) \bn}) \cdot [\bv_h] ds\right|\\
  &\le &C h \sum_{T} \|\bsigma(\bu)\cdot\bn\|_{H^1(T)}  |\bv_h|_{1,T}  \\
  &\le &C h (\|\bu\|_{\wtilde{H}^2(\Omega)} +\lambda_M \|\Div \bu\|_{\wtilde{H}^1(\Omega)}) \|\bv_h\|_{a_h}
\end{eqnarray*} by Corollary \ref{H1P-Korns2}.

\end{proof}

Now we are ready to prove the $H^1$-error estimate.
\begin{thrm} \label{H1error}
Let $\bu$ (resp. $\bu_h$) be the  solution of (\ref{TheprimalEq})(resp. (\ref{Discrete})).
Under the assumption that  $\sigma(\bu)\cdot\bn\in (H^1(\Omega))^2$,  we have
$$\|\bu-\bu_h\|_{a_h} \leq Ch \left(\|\bu\|_{\wtilde{H}^2(\Omega)} +\lambda_M \|\Div \bu\|_{\wtilde{H}^1(\Omega)} \right). $$
\end{thrm}

\begin{proof}
By triangular inequality, we have
$$\|\bu-\bu_h\|_{a_h} \leq \|\bu_h - I_h\bu\|_{a_h} + \|\bu - I_h\bu\|_{a_h}.$$
We have
\begin{eqnarray*}
 \|\bu_h-I_h\bu\|^2_{a_h} &=& a_h(\bu_h - I_h\bu,\bu_h -I_h\bu) \\
&=& a_h(\bu-I_h\bu,\bu_h-I_h\bu) + a_h(\bu_h - \bu,\bu_h -I_h\bu)\\
&\leq& \|\bu_h-I_h\bu\|_{a_h}\|\bu-I_h\bu\|_{a_h} + Ch R(\bu)\|\bu_h-I_h\bu\|_{a_h},
\end{eqnarray*} by Lemma \ref{consistency-term}.
So we have
$$\|\bu_h-I_h\bu\|_{a_h}\leq  \| \bu-I_h\bu\|_{a_h} + Ch R(\bu). $$
Finally, by Proposition \ref{1-a-norm} we have
\begin{eqnarray*}\|\bu-\bu_h\|_{a_h} &\leq& 2\|\bu-I_h\bu\|_{a_h} + ChR(\bu)
  \leq  C_1 h R(\bu). \end{eqnarray*}
\end{proof}

\begin{remark} 
  If the extra regularity $\|\bu\|_{\wtilde{H}^2(\Omega)} +\lambda_M \|\Div \bu\|_{\wtilde{H}^1(\Omega)} \le C\|f\|_{0} $ holds, then the result of Theorem \ref{H1error} improves to
$$\|\bu-\bu_h\|_{a_h}\le  Ch\|f\|_{0}.$$  This would mean that our estimate holds uniformly when $\lambda\to \infty$.
Furthermore, by standard duality argument, we can obtain $L^2$- error estimate of the form:
\begin{eqnarray*}\|\bu-\bu_h\|_{0} &\leq& C h^2\|f\|_{0}.
 \end{eqnarray*}
\end{remark}

\section{Numerical results}

In this section we present numerical examples. The domain is  $\Omega=(-1,1)\times(-1,1)$.
The interface is the zero set of $L(x,y)=x^2 + y^2 - r_0^2$. Let $\Omega^+=\Omega\cap \{(x,y)| L(x,y)>0 \} $,
$\Omega^-=\Omega\cap \{(x,y)| L(x,y)<0 \}$.
The exact solution is chosen as
$$\bu =\left(\frac{1}{\mu}(x^2 + y^2 - r_0^2)x, \frac{1}{\mu}(x^2+y^2-r_0^2)y\right)$$
with various values of $\mu$ and $\lambda$.
For numerical simulation we
partition the domain into uniform right triangles having size $h=2^{-k}, k=3,4,\cdots$.
 \begin{example} In this example, we test two sets of parameters and radii of the interface.
 \begin{enumerate}
 \item   We choose $\mu^- = 1,~\mu^+=100 , \lambda = 5\mu$ and $r_0 = 0.36$.
 \item  We choose $\mu^- = 1,~\mu^+=10 , \lambda = 5\mu$ and $r_0 = 0.48$.
 \end{enumerate}

Tables \ref{tab:1_100_5} and \ref{tab:1_10_5_2} show the convergence behavior of our numerical schemes for both examples.
In both cases, we see the optimal order of convergence in $L^2$, $H^1$ and divergence norms.
$x$-components of the solution are plotted in Figures \ref{fig:1_100_5} and  \ref{fig:1_10_5}.
\end{example}
\begin{table}[h]
\begin{center}
\begin{tabular}{|c||c|c||c|c||c|c|}
\hline   $1/h$& \small{$\|\bu-\bu_h\|_{0}$} & order & \small{$\|\bu-\bu_h\|_{1,h}$} & order & \small{$\|\Div \bu - \Div \bu_h\|_{0}$} & order\\
\hline 
          8  & 1.887e-3 &  & 4.098e-2 &  & 4.694e-2 & \\ 
           16  & 5.354e-4 & 1.817 & 1.957e-2 & 1.066 & 2.311e-2 & 1.022\\ 
           32  & 1.186e-4 & 2.175 & 9.547e-3 & 1.036 & 1.089e-2 & 1.085\\ 
            64  & 2.864e-5 & 2.050 & 4.850e-3 & 0.977 & 5.382e-3 & 1.017\\ 
           128  & 6.793e-6 & 2.076 & 2.430e-3 & 0.997 & 2.637e-3 & 1.029\\ 
           256  & 1.673e-6 & 2.021 & 1.217e-3 & 0.998 & 1.310e-3 & 1.009\\ 
          \hline
\end{tabular}
\caption{ $\mu^- = 1,~\mu^+=100 , \lambda = 5\mu$}
\label{tab:1_100_5}
\end{center}
\end{table}


\begin{table}[hbt]
\begin{center}
\begin{tabular}{|c||c|c||c|c||c|c|}
\hline   $1/h$& \small{$\|\bu-\bu_h\|_{0}$} & order & \small{$\|\bu-\bu_h\|_{1,h}$} & order & \small{$\|\Div \bu - \Div \bu_h\|_{0}$} & order\\
\hline
             8  & 2.910e-3 &  & 7.972e-2 &  & 8.598e-2 & \\ 
            16  & 7.450e-4 & 1.966 & 3.822e-2 & 1.061 & 4.155e-2 & 1.049\\ 
            32  & 1.841e-4 & 2.017 & 1.942e-2 & 0.977 & 2.091e-2 & 0.991\\ 
            64  & 4.606e-5 & 1.999 & 9.787e-3 & 0.989 & 1.049e-2 & 0.996\\ 
           128  & 1.143e-5 & 2.010 & 4.920e-3 & 0.992 & 5.255e-3 & 0.997\\ 
           256  & 2.851e-6 & 2.004 & 2.466e-3 & 0.997 & 2.630e-3 & 0.999\\ 
          \hline
\end{tabular}
\caption{ $\mu^- = 1,~\mu^+=10 , \lambda = 5\mu$}
\label{tab:1_10_5_2}
\end{center}
\end{table}

\begin{example}[Nearly incompressible case]
\begin{enumerate}
 \item We let $\mu^- = 1,~\mu^+=10, \lambda = 100\mu, \nu=0.495$ and $r_0 = 0.7$.
 \item We let $\mu^- = 1,~\mu^+=10, \lambda = 1000\mu ,\nu=0.4995$ and $r_0 = 0.6$.
 \end{enumerate}
\begin{table}[hbt]
\begin{center}
\begin{tabular}{|c||c|c||c|c||c|c|}
\hline   $1/h$& \small{$\|\bu-\bu_h\|_{0}$} & order & \small{$\|\bu-\bu_h\|_{1,h}$} & order & \small{$\|\Div \bu - \Div \bu_h\|_{0}$} & order\\
\hline
             8  & 7.733e-3 &  & 1.456e-1 & & 2.132e-1 & \\ 
            16  & 2.487e-3 & 1.644 & 7.541e-2 & 0.949 & 1.136e-1 & 0.909\\ 
            32  & 7.434e-4 & 1.742 & 3.729e-2 & 1.016 & 5.527e-2 & 1.039\\ 
            64  & 2.124e-4 & 1.807 & 1.876e-2 & 0.991 & 2.730e-2 & 1.018\\ 
           128  & 5.508e-5 & 1.948 & 9.417e-3 & 0.994 & 1.347e-2 & 1.019\\ 
           256  & 1.428e-5 & 1.948 & 4.719e-3 & 0.997 & 6.686e-3 & 1.011\\ 
          \hline
\end{tabular}
\caption{ $\mu^- = 1,~\mu^+=10 , \lambda = 100\mu$}
\label{tab:1_10_100}
\end{center}
\end{table}


\begin{table}[hbt]
\begin{center}
\begin{tabular}{|c||c|c||c|c||c|c|}
\hline  $1/h$& \small{$\|\bu-\bu_h\|_{0}$} & order & \small{$\|\bu-\bu_h\|_{1,h}$} & order & \small{$\|\Div \bu - \Div \bu_h\|_{0}$} & order\\
\hline
           8  & 7.655e-2 &  & 1.125e-1 &  & 1.628e-0 & \\ 
            16  & 2.372e-2 & 1.690 & 5.570e-2 & 1.014 & 9.065e-1 & 0.846\\ 
            32  & 6.806e-2 & 1.801 & 2.829e-2 & 0.978 & 4.518e-1 & 1.004\\ 
            64  & 1.847e-3 & 1.882 & 1.417e-2 & 0.997 & 2.247e-1 & 1.008\\ 
           128  & 4.811e-4 & 1.941 & 7.110e-3 & 0.995 & 1.111e-1 & 1.016\\ 
           256  & 1.230e-4 & 1.968 & 3.563e-3 & 0.997 & 5.534e-2 & 1.006\\ 
          \hline
\end{tabular}
\caption{ $\mu^- = 1,~\mu^+=10 , \lambda = 1000\mu$}
 \label{tab:1_10_1000_2}
\end{center}
\end{table}


 Tables \ref{tab:1_10_100}  and \ref{tab:1_10_1000_2}  show the convergence behavior.
In both cases, we see the optimal order of convergence in $L^2$, $H^1$ and divergence norms.
No locking phenomena occurs in both cases.
Again $x$-components of the solution  are plotted in Figures \ref{fig:1_10_100} and  \ref{fig:1_10_1000}.
\end{example}

\begin{example}[Ellipse interface case]
Next we consider examples with elliptic shaped interface. The domain is the same as above, and the
interface is represented by $L(x,y)=\frac{x^2}{4} + y^2 - r_0^2 = 0$.
The exact solution is chosen as
$$\bu =\left(\frac{1}{\mu}(\frac{x^2}{4} + y^2 - r_0^2)x, \frac{1}{\mu}(\frac{x^2}{4}+y^2-r_0^2)y\right)$$
with various values of $\mu$ and $\lambda$.
\begin{enumerate}
 \item We let $\mu^- = 1,~\mu^+=10, \lambda = 5\mu , r_0 = 0.4 $.
 \item  We let $\mu^- = 1,~\mu^+=100, \lambda = 5\mu , r_0 = 0.3 $.
\end{enumerate}

\begin{table}[hbt]
\begin{center}
\begin{tabular}{|c||c|c||c|c||c|c|}
\hline  $1/h$& \small{$\|\bu-\bu_h\|_{0}$} & order & \small{$\|\bu-\bu_h\|_{1,h}$} & order & \small{$\|\Div \bu - \Div \bu_h\|_{0}$} & order\\
\hline
             8  & 2.477e-3 &  & 5.920e-2 &  & 6.744e-2 & \\ 
            16  & 6.689e-4 & 1.888 & 2.909e-2 & 1.025 & 3.340e-2 & 1.014\\ 
            32  & 1.704e-4 & 1.973 & 1.480e-2 & 0.975 & 1.694e-2 & 0.979\\ 
            64  & 4.200e-5 & 2.020 & 7.485e-3 & 0.983 & 8.531e-3 & 0.990\\ 
           128  & 1.029e-5 & 2.029 & 3.765e-3 & 0.992 & 4.281e-3 & 0.995\\ 
           256  & 2.579e-6 & 1.996 & 1.886e-3 & 0.997 & 2.144e-3 & 0.998\\ 
          \hline
\end{tabular}
\caption{ $\mu^- = 1,~\mu^+=10 , \lambda = 5\mu,$ elliptical interface}
\label{tab:1_10_5_ellipse}
\end{center}
\end{table}


\begin{table}[hbt]
\begin{center}
\begin{tabular}{|c||c|c||c|c||c|c|}
\hline   $1/h$& \small{$\|\bu-\bu_h\|_{0}$} & order & \small{$\|\bu-\bu_h\|_{1,h}$} & order & \small{$\|\Div \bu - \Div \bu_h\|_{0}$} & order\\
\hline
   8  & 2.018e-3 &  & 3.164e-2 &  & 3.788e-2 & \\ 
   16  & 6.644e-4 & 1.647 & 1.424e-2 & 1.151 & 2.066e-2 & 0.875\\ 
   32  & 1.376e-4 & 2.227 & 7.314e-3 & 0.962 & 9.592e-3 & 1.107\\ 
   64  & 2.736e-5 & 2.330 & 3.735e-3 & 0.969 & 4.458e-3 & 1.105\\ 
   128  & 6.896e-6 & 1.988 & 1.880e-3 & 0.991 & 2.229e-3 & 1.000\\ 
   256  & 1.726e-6 & 1.998 & 9.434e-4 & 0.994 & 1.107e-3 & 1.010\\ 
 \hline
\end{tabular}
\caption{ $\mu^- = 1,~\mu^+=100 , \lambda = 5\mu,$ elliptical interface}
\label{tab:1_100_5_ellipse}
\end{center}
\end{table}


 Tables \ref{tab:1_10_5_ellipse}  and \ref{tab:1_100_5_ellipse}  show the convergence behavior.
We observe similar optimal convergence rates for all norms.
 Figures \ref{fig:1_10_5_ellipse}  and \ref{fig:1_100_5_ellipse} show the $x$-components of the solution.
\end{example}

\begin{example}[Unknown solution]
This last example computes a problem with unknown solution.
We choose $\mu^- = 1,~\mu^+=100, \nu^- = 0.28, \nu^+ = 0.4 , r_0 = 0.3 $ and
 $\bF = \left(-\frac{11}{4} - \frac{\lambda}{\mu}x, -\frac{29}{4} - \frac{\lambda}{\mu}y\right)$ with the same elliptical interface as in the previous example.

Figure \ref{fig:unknown_sol} shows the x-component of the computed solution.
\end{example}
\section{Conclusion}

In the present work, we have developed a new finite element method for solving planar elasticity problems with an interface along which distinct materials are bonded.
The methods are based on the IFEM using CR element modified to satisfy  Laplace-Young condition along the interface.
Our methods yield smaller matrix size than XFEM since we do not use any extra dofs other than the
edge based functions.
The jump terms along the edges are added to ensure the stability of the scheme.

We have proved an interpolation error in $H^1$ and  $H(div)$ norm (with $\sqrt{\lambda}$ factor).
 For the error estimate of $\bu-\bu_h$, we have obtained an optimal $O(h)$ error in $H^1$ and $H(div)$ norm under the regularity that  $\bsigma(\bu)\in (H^1(\Omega))^2$.
 The numerical tests show the optimal $O(h)$ error in $H^1$ norm, and  $O(h^2)$ in $L^2$ norm.

As future works, we will consider problems with nonhomogeneous jump conditions and three dimensional problems.

\begin{appendices}
We sketch the proof of  Proposition \ref{prop:exiten-CRbasis}.
The degrees of freedom and point continuity of (\ref{eq:dof1}) give rise to
the ten equations for the coefficients of $\hat{\phi}_1$ and $\hat{\phi}_2$,
in the form
\begin{equation}\label{eq:Amatri0}
\begin{pmatrix}
  A  &  \mathbf 0 \\
 \mathbf 0 & A
 \end{pmatrix} \begin{pmatrix}
  \bc_1 \\  \bc_2 \end{pmatrix} =\begin{pmatrix}
  \bg_1  \\  \bg_2 \end{pmatrix}
\end{equation} where
\begin{equation}\label{eq:Amatrix1}
 {A}  = \left(
\begin{array}{cccccc}
  1 & \frac 12 & \frac 12 & 0 & 0 & 0 \\
  1-y & 0 & \frac12 (1-y^2) & y & 0 & \frac12 y^2 \\
  1-x & \frac12 (1-x^2) & 0 & x & \frac12 x^2 & 0 \\
  -1 & -x & 0 & 1 & x & 0 \\
  -1 & 0 & -y & 1 & 0 & y
\end{array}
\right)
\end{equation} and  $\bc_i=(a^+_i,b^+_i,c^+_i,a^-_i,b^-_i,c^-_i)$, $i=1,2$ are the vector of the unknowns.
The jump conditions along the interface (last equations of  (\ref{eq:dof1})) give rise to the following equations.
 \begin{eqnarray}\label{eq:Amatr_jump} \mu^+\begin{pmatrix} b_{1}^+ & c_{1}^+ \\ b_{2}^+ &c_{2}^+ \end{pmatrix}\cdot\bn-\mu^- \begin{pmatrix} b_{1}^- & c_{1}^- \\ b_{2}^-&c_{2}^-\end{pmatrix}\cdot\bn
 + \mu^+\begin{pmatrix} b^+_1 & b^+_2 \\
 c^+_1 & c^+_2\end{pmatrix}\cdot\bn
 - \mu^-\begin{pmatrix} b^-_1 & b^-_2 \\
 c^-_1 & c^-_2\end{pmatrix}\cdot\bn \\
+ \lambda^+ (b^+_1+ c^+_2)\bn-\lambda^-(b^-_1 + c^-_2)\bn=0.\nonumber
 \end{eqnarray}
Combining  (\ref{eq:Amatri0}) and (\ref{eq:Amatr_jump}), we get the following system of twelve equations in
twelve unknowns.
\begin{equation}
 {M} =
\begin{pmatrix}
  A  &  \mathbf 0 \\ \mathbf d_1^T & \mathbf d_2^T \\
 \mathbf 0 & A    \\ \be_1^T & \be_2^T
 \end{pmatrix}  \begin{pmatrix}
  \bc_1 \\  \bc_2 \end{pmatrix} =\begin{pmatrix}
 \bg_1   \\ 0   \\  \bg_2  \\ 0\end{pmatrix} .
\end{equation}
Here
\begin{eqnarray*}
\mathbf d_1^T &= & \begin{pmatrix} 0,(2\mu^+ +\lambda^+)n_1,\mu^+n_2 ,0,-(2\mu^- +\lambda^-) n_1,-\mu^-n_2\end{pmatrix}:=(d_{1i})_{i=1}^6, \\
\mathbf d_2^T&= &\begin{pmatrix}  0, \mu^+ n_2,\lambda^+n_1, 0,- \mu^-n_2,-\lambda^-n_1\end{pmatrix}:=(d_{2i})_{i=1}^6\\
\be_1 ^T&= & \begin{pmatrix} 0, \lambda^+ n_2,\mu^+n_1, 0,-\lambda^-n_2,-\mu^-n_1\end{pmatrix}:=(e_{1i})_{i=1}^6\\
\be_2^T &= &\begin{pmatrix} 0, \mu^+n_1,(2\mu^+ +\lambda^+)n_2,0,-\mu^-n_1,-(2\mu^- +\lambda^-) n_2\end{pmatrix}:=(e_{2i})_{i=1}^6.
\end{eqnarray*}
Now we will compute the determinant of   $M$.
Adding columns 6,5,4 to 3,2,1 and columns 12,11,10 to 9,8,7 (resp.), and by row  eliminations,
we obtain following.
 \begin{equation}\label{eq:app_basic-02}
M' :=
\begin{pmatrix}
 U   &    0  & |&  {\bf O}  & 0 \\
 0   & \bar d_{66}& |&  \bar {\mathbf d_2}^T   & 0\\
 --- &    -  & |&  ---- & -\\
 {\bf O} &    0  & |&  U    & 0 \\
 0   & \bar e_{16}& |&  0    &\bar e_{66}
 \end{pmatrix},\mbox{ where } U= \begin{pmatrix}
  1 & \frac 12 & \frac 12 & 0 & 0  \\
  0 &- \frac 12 & 0 & y & 0  \\
  0 & 0  &-\frac 12 & x & \frac12 x^2  \\
  0 & 0 & 0 & 1 & x   \\
  0 & 0 & 0 & 0 &-x
\end{pmatrix}
\end{equation}
Here $\bar d_{66}, \bar e_{16}$ and $\bar e_{66}$ are given by
\begin{eqnarray*}
\label{eq:app_e16}
x \bar d_{66}&=&-n_1 y\{(2\mu^+ +\lambda^+)x y+(2\mu^- +\lambda^-)(1-xy)\}-xn_2\{\mu^+ xy+ \mu^-(1-xy)\},\\
x \bar e_{16}&=&-n_2y \{\lambda^+xy+\lambda^-(1-xy)\} -n_1x \{ \mu^+xy +\mu^-(1-xy)\},\\
x \bar e_{66}&=&- yn_1\{\mu^+ xy +\mu^-(1-xy)\}-xn_2\{(2\mu^- +\lambda^-)(1-xy)+(2\mu^+ +\lambda^+)xy\}.
\end{eqnarray*}
\begin{lmm} The determinant of matrix $M'$ is given  as follows.
\begin{eqnarray}\label{eq:Det-Mpr}
det(M') 
=\frac1 {16}\{ x\bar d_{66} x\bar e_{66}-4 x\bar e_{16} \cdot cofac\}.
\end{eqnarray}
Here, with the notation $[\lambda]=\lambda^+-\lambda^-,  [\lambda]= \lambda^+-\lambda^-$, cofac is given by
\begin{eqnarray*}
cofac&=&-\frac14 ([\mu]  n_2 xy^2+[\lambda]n_1 x^2y+ x\lambda^-n_1 +y\mu^-n_2).
\end{eqnarray*}
\end{lmm}
\begin{proof}
 This can be obtained by expanding the determinant with resp. to fifth column of  $M'$.
\end{proof}
\begin{proposition} The determinant of matrix $M'$ is always negative.
\end{proposition}
\begin{proof}
 Substituting  $(n_1,n_2)= \frac{(y,x)}{\sqrt{x^2+y^2}}$  into (\ref{eq:Det-Mpr}) we see
\begin{eqnarray*}
-16\sqrt{x^2+y^2}\cdot det(M')& = &\left\{ y^2\{(2\mu^+ +\lambda^+)x y+(2\mu^- +\lambda^-)(1-xy)\}+x^2\{\mu^+ xy+ \mu^-(1-xy)\} \right\} \\
&&   \cdot \left\{y^2\{\mu^+ xy +\mu^-(1-xy)\}+x^2\{(2\mu^- +\lambda^-)(1-xy)+(2\mu^+ +\lambda^+) xy\} \right\}\\
&&    + \bar e_{16} x (([\mu]+[\lambda]) x^2y^2+ ( \lambda^-+\mu^-)xy ) \\
&=&  \left\{ y^2([2\mu  +\lambda ]x y+(2\mu^- +\lambda^-))+x^2([\mu] xy+ \mu^-) \right\} \\
&& \cdot \left\{y^2([\mu] xy +\mu^-)+x^2([2\mu +\lambda] xy+(2 \mu^- +\lambda^-)) \right\}\\
&& - (xy)^2(([\mu]+[\lambda]) xy+ (\lambda^-+\mu^-))^2\\
&=& y^4A(2A+B)+x^4A(2A+B)+x^2y^2\{(2A+B)^2+A^2\}-(xy)^2(A+B)^2>0
\end{eqnarray*}
where $A=[\mu]xy+\mu^- , B= [\lambda] xy+ \lambda^- $.
Hence the determinant is always  negative.
\end{proof}
\end{appendices}

\begin{thebibliography}{999} 
\bibitem{Arnold-IP} D. N. Arnold, {\it An interior penalty finite element method with discontinuous elements}, SIAM
J. Numer. Anal., 19 (1982), pp. 742--760.
\bibitem{Arnold-Wint} D. N. Arnold and R. Winther, {\it
  Mixed finite elements for elasticity. Numer. Math.}, 92(3) (2002), pp. 401--419.
\bibitem{Ar-B-Co-Ma} D. Arnold, F. Brezzi, B. Cockburn, and D. Marini, {\it Discontinuous Galerkin methods for
elliptic problems}, in Discontinuous Galerkin Methods. Theory, Computation and Applications,
B. Cockburn, G. E. Karniadakis, and C.-W. Shu, eds., Lecture Notes in Comput.
Sci. Engrg. 11, Springer-Verlag, NewYork, 2000, pp. 89--101.

\bibitem{Babuska-Suri1} I. Babuska and M. Suri,  {\it Locking effect in the finite element approximation of elasticity problem}, Numer. Math. 62 (1992), pp. 439--463.
\bibitem{Babuska-Suri2} I. Babuska and M. Suri,  {\it On locking and robustness in the finie element method},
 SIAM J. Numer. Anal. 29 (1992), pp. 1261--1293.

\bibitem{BeckerB.Hansbo} R. Becker, E. Burman,  P. Hansbo, {\it A Nitsche extended finite element method for incompressible elasticity with discontinuous modulus of elasticity}, Comput. Methods Appl. Mech. Engrg.  198  (2009), pp. 3352--3360.
\bibitem{Belytschko_Parimi} T. Belytschko, C. Parimi, N. Mo\"es, N. Sukumar, S. Usui, {\it Structured extended finite element methods for solids defined by implicit surfaces}, Int. J. Numer. Meth. Engrg. 56 (2003), pp. 609--635.
\bibitem{Belytschko_Black} T. Belytschko, T. Black, {\it Elastic crack growth in finite elements
with minimal remeshing}, Int. J. Numer. Meth. Engrg. 45 (1999), pp. 601--620.
%
\bibitem{Brenner-Korn} S. C. Brenner, {\it Korn's inequalities for piecewise $H^1$ vector fields}, Math. Comp. V. 72, No 274 (2003), pp. 1067--1087.
\bibitem{Brenner_Sung} S. C. Brenner and L. Y. Sung, {\it Linear finite element methods for planar
linear elasticity},
   Math. Comp, V. 59, No 200, (1992), pp. 321--338.

\bibitem{BrezziFortin} F. Brezzi and M. Fortin, {\it Mixed and
hybrid finite element methods}, Springer-Verlag, New-York, 1991.
\bibitem{Chang_Kwak} K. S. Chang and Do Y. Kwak, {\it  Discontinuous Bubble scheme for elliptic problems  with jumps in the solution}, Comp. Meth. Appl. Mech. Engrg. V. 200 (2011), pp. 494--508.

\bibitem{Ch-Kw-We} S. H. Chou, D. Y. Kwak and K. T. Wee, {\it  Optimal convergence analysis of an immersed
interface finite element method},  Adv. Comput. Math, V. 33 (2010), pp. 149--168.
\bibitem{Ciarlet} P. G. Ciarlet, {\it The finite element method for elliptic problems}, North Holland,  1978.
\bibitem{Ciarlet-Elast} P. G. Ciarlet, {\it mathematical elasticity Vol I}, North Holland,  1988.
\bibitem{Crouzeix} M. Crouzeix and P. A. Raviart, {\it Conforming and nonconforming finite element
  methods for solving the stationary Stokes equations}, RAIRO Anal. Num\'{e}r. V.7 (1973), pp. 33--75.
%
\bibitem{Falk} R. S. Falk, {\it Nonconforming Finite Element Methods for the Equations of Linear Elasticity},
  Mathematics of Computation, Vol. 57, No. 196 (1991), pp. 529--550.
\bibitem{A.P.Hansbo2002} A. Hansbo and P. Hansbo, {\it An unfitted finite element method, based on Nitsche's
method, for elliptic interface problems},
Comput. Methods Appl. Mech. Engrg. 191, (2002), pp. 5537--5552.
\bibitem{A.P.Hansbo2004} A. Hansbo and P. Hansbo, {\it A finite element method for the simulation of strong
and weak discontinuities in solid mechanics}, Comput. Methods Appl. Mech. Engrg. 193 (2004) pp. 3523--3540.
%

\bibitem{P.Hansbo_Lar2002} P. Hansbo and M. G. Larson,
{\it Discontinuous  Galerkin and the Crouzeix-Raviart element: Applications to elasticity},
Mathematical Modelling and Numerical Analysis, ESAIM, Vol. 37 (2003), pp. 63--72.
 

\bibitem{Krysl-Bely} Petr Krysl and Ted Belytschko,
{\it An effcient linear-precision partition of unity basis for
unstructured meshless methods}, Commun. Numer. Meth. Engng. 16 (2000), pp.239--255.

 \bibitem{Kwak-Jin} Do Y. Kwak and Sangwon Jin, {\it A stabilized P1 immersed finite element method for the interface elasticity problems}, arXiv:1408.4227, Aug 2014.
\bibitem{Kwak-We-Ch}  Do Y. Kwak, K. T. Wee and K. S.  Chang,
{\it An analysis of a broken $P_1$ -nonconforming finite element method
for interface problems}, SIAM J. Numer. Anal. \textbf{48} (2010), pp. 2117--2134.
 
\bibitem{Lai_Li_L} M. Lai, Z.  Li and X. Lin,
{\it Fast solvers for 3D Poisson equations involving interfaces in a finite or
the infinite domain}, J. Comput. Appl. Math. \textbf{191} (2006), no. 1, pp.
106--125.
\bibitem{Legrain} G. Legrain, N. Moes  and E. Verron, {\it Stress analysis around crack tips in finite strain problems using the eXtended finite element method}, Int. J. Numer. Meth. Engng 63, (2005), pp. 290--314.
 
\bibitem{Leguillon}  D. Leguillon, E. Sanchez-Palencia, {\it Computation of Singular Solutions in Elliptic Problems and Elasticity}, Wiley, 1987.
 \bibitem{LeVequeLi} R. J. LeVeque and Z. Li,  {\it Immersed interface method for Stokes flowwith elastic boundaries or surface tension}, SIAM J. Sci. Comput. 18 (1997), pp. 709--735.
\bibitem{Li1994} R. J. LeVeque and Z. Li, {\it The immersed interface method for elliptic equations with
  discontinuous coefficients and singular sources}, SIAM J. Numer. Anal. \textbf{31} (1994), pp. 1019--1044.
%

\bibitem{Li2004} Z. Li, T. Lin, Y. Lin and R. C. Rogers, {\it An immersed finite element space and its
  approximation capability}, Numer. Methods. Partial Differential Equations \textbf{20}
   (2004), pp. 338--367.
\bibitem{Li2003} Z. Li, T. Lin and X. Wu, {\it New Cartesian grid methods for interface problems using the
  finite element formulation}, Numer. Math. \textbf{96} (2003), pp. 61--98.
%
\bibitem{Lin-Zhang-Sh} T. Lin, D.  Sheen, X. Zhang,
{\it A locking-free immersed finite element method for planar elasticity interface problems}, J. Comput.  Phys., \textbf{247}(2013) pp. 228--247.
\bibitem{Moes:99} N. Mo\"es, J. Dolbow and T. Belytschko,   {\it A finite element method for crack
growth without remeshing}, Int. J. Numer. Methods Eng. 46(1)  (1999), pp.131--156.
\bibitem{Nitsche}  J. Nitsche,  {\it \"Uber ein Variationsprinzip zur L\"osung von Dirichlet-Problemen bei
Verwendung von Teilraumen die keinen Randbedingungen unterworfen sind},
Abh. Math. Sem. Univ. Hamburg 36 (1971) pp. 9--15.
\bibitem{Oeverm_S_K}
M. Oevermann, C. Scharfenberg, R. Klein, {\it  A sharp interface finite volume method for elliptic equations on Cartesian grids}, J. Comput. Phys., \textbf{228} (2009) No. 14, pp. 5184--5206.
\bibitem{Wheeler} M.F. Wheeler, {\it An elliptic collocation-finite element method with interior penalties}, SIAM J. Numer. Anal.  15  (1978), pp. 152--161.
\end{thebibliography}
\end{document}